\documentclass[letterpaper, 10 pt, conference]{ieeeconf}  

\IEEEoverridecommandlockouts                              

\usepackage{amsmath}
\usepackage{color}
\usepackage{amssymb}  
\usepackage{amsfonts}
\usepackage{graphicx}
\usepackage{dsfont}
\usepackage{psfrag}
\usepackage{epstopdf}
\usepackage{mathtools}
\usepackage{cuted}
\usepackage{balance}

\usepackage[noadjust]{cite}
\usepackage{bbm}
\usepackage{algorithm}
\usepackage{algorithmic}
\usepackage{savesym}
\usepackage{amsthm}
\usepackage{breqn}

\theoremstyle{definition}
\newtheorem{definition}{Definition}

\newtheorem{remark}{Remark}

\theoremstyle{plain}
\newtheorem{theorem}{Theorem}

\newtheorem{probstat}{Problem}

\DeclareMathOperator*{\argmin}{argmin}

\title{\LARGE \bf
On Learning Discrete-Time Fractional-Order Dynamical Systems
}

\author{Sarthak Chatterjee$^{\dag}$ \qquad S\'ergio Pequito$^{\ddag}$
\thanks{$^{\dag}$Sarthak Chatterjee is with the department of Electrical, Computer, and Systems Engineering, Rensselaer Polytechnic Institute, Troy, NY, USA
        {\tt\small sarthak.chatterjee92@gmail.com}}%
\thanks{$^{\ddag}$S\'ergio Pequito is with the Delft Center for Systems and Control, Delft University of Technology, Delft, The Netherlands
        {\tt\small Sergio.Pequito@tudelft.nl}}%
}

\begin{document}

\maketitle
\thispagestyle{empty}
\pagestyle{empty}

\begin{abstract}

Discrete-time fractional-order dynamical systems (DT-FODS) have found innumerable applications in the context of modeling spatiotemporal behaviors associated with \mbox{long-term} memory. Applications include neurophysiological signals such as electroencephalogram (EEG) and electrocorticogram (ECoG). Although learning the spatiotemporal parameters of DT-FODS is not a new problem, when dealing with neurophysiological signals we need to guarantee performance standards. Therefore, we need to understand the trade-offs between sample complexity and estimation accuracy of the system parameters. Simply speaking, we need to address the question of how many measurements we need to collect to identify the system parameters up to an uncertainty level. In this paper, we address the problem of identifying the spatial and temporal parameters of DT-FODS. The main result is the first result on non-asymptotic finite-sample complexity guarantees of identifying \mbox{DT-FODS}. Finally, we provide evidence of the efficacy of our method in the context of forecasting real-life intracranial EEG time series collected from patients undergoing epileptic seizures.

\end{abstract}

\section{Introduction}

Learning a model (or system identification) is the problem of estimating the parameters of a dynamical system given an input-output time series of the trajectories that have been generated by the former. In many problem settings in control theory, time series analysis, reinforcement learning, and econometrics, estimating the parameters of a system from its \mbox{input-output} behavior is an important problem, especially in the absence of analytical tools to find the underlying dynamics. Extensive attention has been paid to the topic, with~\cite{ljung1999sysid} and~\cite{soderstrom1988system} being comprehensive references on the same.



Lately, there has been a renewed interest in investigating finite-sample guarantees for problems in classical system identification and control theory from the lens of statistical learning theory. On the one hand, significant work has been done on developing finite-time regret guarantees for the linear-quadratic regulator when the latter is trying to control a system with unknown dynamics~\cite{abbasi2011regret,dean2018regret,dean2019sample,dean2019safely,mania2019certainty,cohen2019learning}. Alternatively, \mbox{finite-sample} guarantees have also been derived in the context of high probability bounds with respect to the norm of the estimation error of the system's parameters in a wide variety of problem settings~\cite{tu2017non,faradonbeh2018finite,hazan2018spectral,hardt2018gradient,sarkar2019near,sarkar2021finite,oymak2019non,tsiamis2019finite,simchowitz2018learning,simchowitz2019learning,wagenmaker2020active,jedra2019sample,jedra2020finite,zheng2021non}.

Nonetheless, most of the work focuses on cases where there is a Markovian dependence of the current state of a system on just the previous state, which is insufficient in describing the \mbox{long-term} behavior of the aforementioned systems. Yet, systems that we often encounter in real-life demonstrate phenomena such as hysteresis and long-term memory, in which the current system state is dependent on a combination of several past states or the entire gamut of states seen so far in time.

In particular, fractional-order dynamical systems (FODS) have attracted a lot of interest in recent years for being able to successfully model a wide variety of dynamical system behaviors, chief among them being (nonexponential) \mbox{power-law} decay in the dependence of the current state on past states, systems exhibiting long-term memory or fractal properties, or dynamics where there are adaptations across multiple time scales~\cite{moon2008chaotic,lundstrom2008fractional,werner2010fractals,thurner2003scaling,teich1997fractal}. FODS have been used in domains as disparate as biological swarms~\cite{west2014networks}, chaotic systems~\cite{petravs2011fractional}, gaseous dynamics~\cite{chen2010anomalous}, and cyber-physical systems~\cite{xuecps}, to mention a few. In the context of this work, we emphasize the applications to neuromodulation as well as modeling neurophysiological signals such as electroencephalogram (EEG) or electrocorticogram (ECoG)~\cite{romero2020fractional,chatterjee2020fractional,magin2006fractional}.

In this paper, we will investigate the problem of learning the spatial and temporal parameters of a discrete-time fractional-order dynamical system (DT-FODS). We provide \mbox{finite-sample} and iteration complexity guarantees for the same, and we further illustrate the working of our proposed approach on real-life intracranial EEG (iEEG) data collected from patients undergoing epileptic seizures. In the context of using model-based approaches for neurostimulation (such as to perform epileptic seizure suppression~\cite{chatterjee2020fractional}), robust guarantees for the iteration and finite-sample complexities of learning the model parameters is of paramount importance. Additionally, using approaches such as model predictive control (MPC) in order to develop neurostimulation strategies often require certifications on the minimum amount of data required in order to learn robust models.

The rest of the paper is structured as follows. Section II introduces DT-FODS and the problem that we seek to address in this paper. Section III outlines an approach to learn the parameters of a DT-FODS as well as the main results in terms of the finite-sample and iteration complexity guarantees for the same. Section IV demonstrates the working of our approach on real-life iEEG data. Section V concludes the paper.


\section{Problem Statement}

We consider a linear discrete-time fractional-order dynamical system described as follows
\begin{align}\label{eq:frac_model}
    \Delta^\alpha x[k+1] &= Ax[k] + w[k] \nonumber \\
    x[0] &= x_0,
\end{align}
where $x[k] \in \mathbb{R}^n$ is the \emph{state} for time step $k \in \mathbb{N}$ and $A \in \mathbb{R}^{n \times n}$ is the \emph{system matrix}. The \emph{process noise} $w[k]$ is assumed to be an i.i.d. additive white noise process originating from the standard normal distribution, i.e., \mbox{$w[k] \thicksim \mathcal{N}(0,\sigma^2 I_n)$}. The above system model is similar to a classic discrete-time linear time-invariant model except for the inclusion of the Gr\"unwald-Letnikov fractional derivative, whose expansion and discretization for the $i$-th state, \mbox{$1 \leq i \leq n$}, can be written as
\begin{equation}\label{eq:frac_deriv}
    \Delta^{\alpha_i} x_i[k] = \sum_{j=0}^k \psi (\alpha_i,j) x_i[k-j],
\end{equation}
where $\alpha_i$ is the fractional order corresponding to state $i$ (with $\alpha_i > 0$ for all $i$) and
\begin{equation}
    \label{eq:weights}
    \psi(\alpha_i,j) = \frac{\Gamma(j-\alpha_i)}{\Gamma(-\alpha_i) \Gamma(j+1)},
\end{equation}
with $\Gamma(\cdot)$ being the gamma function defined by \mbox{$\Gamma (z) = \int_0^{\infty} s^{z-1} e^{-s} \: \mathrm{d}s$} for all complex numbers $z$ with $\Re (z) > 0$~\cite{DzielinskiFOS}.

Given the above ingredients, we seek to solve the following problem in this paper.

\begin{probstat}
Consider the discrete-time fractional-order dynamical system model given in~\eqref{eq:frac_model} with $w[k]$ being an additive white Gaussian process noise. We seek to provide \mbox{non-asymptotic} finite-sample as well as iteration complexity guarantees of identifying the temporal and spatial system parameters of~\eqref{eq:frac_model} ($\{ \alpha_i \}_{i=1}^n$ and $A$, respectively) using an observed trajectory of the states.
\end{probstat}

\section{Learning Discrete-Time Fractional-Order Dynamical Systems}

In this section, we establish in a sequential manner, a bi-level iterative scheme to learn the spatial and temporal components of a DT-FODS. We first start with some fundamentals regarding \mbox{DT-FODS}.

\subsection{Preliminaries}

Let us first review some essential theory for \mbox{fractional-order} systems, including an approximation of~\eqref{eq:frac_model} as an LTI system. Using the expansion of the Gr\"unwald-Letnikov derivative in~\eqref{eq:frac_deriv}, we have
\begin{align}
    \Delta^\alpha x[k] &= \begin{bmatrix}\Delta^{\alpha_1}x_1[k]\\ \vdots \\ \Delta^{\alpha_n}x_n[k] \end{bmatrix} = \begin{bmatrix}\sum_{j=0}^k\psi(\alpha_1,j)x_1[k-j]\\ \vdots \\ \sum_{j=0}^k\psi(\alpha_n,j)x_n[k-j]\end{bmatrix} \nonumber \\ &= \sum_{j=0}^k\underbrace{\begin{bmatrix}\alpha_1 & \ldots & 0\\ \vdots & \ddots & \vdots \\ 0 & \ldots & \alpha_n\end{bmatrix}}_{D(\alpha,j)}\begin{bmatrix}x_1[k-j] \\ \vdots\\ x_n[k-j]\end{bmatrix} \nonumber \\ &= \sum_{j=0}^k D(\alpha,j)x[k-j].
    \label{eq:Deltaalpha}
\end{align}
The above formulation distinctly highlights one of the main peculiarities of DT-FODS in that the fractional derivative $\Delta^{\alpha} x[k]$ is a weighted linear combination of not just the previous state but of every single state up to the current one, with the weights given by~\eqref{eq:weights} following a power-law decay.

Plugging~\eqref{eq:Deltaalpha} into the DT-FODS formulation~\eqref{eq:frac_model}, we have
\begin{equation}
    \sum_{j=0}^{k+1} D(\alpha,j) x[k+1-j] = Ax[k] + w[k],
\end{equation}
or, equivalently,
\begin{equation}
    D(\alpha,0) x[k+1] = - \sum_{j=1}^{k+1} D(\alpha,j) x[k+1-j] + Ax[k] + w[k],
\end{equation}
which leads to
\begin{equation}
\label{eq:D_evol}
    x[k+1] = -\sum_{j=0}^k D(\alpha,j+1) x[k-j] + Ax[k] + w[k],
\end{equation}
since $D(\alpha,0) = I_n$. Alternatively,~\eqref{eq:D_evol} can be written as
\begin{align}\label{eq:state_evol_2}
    x[k+1] &= \sum_{j=0}^k A_j x[k-j] + w[k] \nonumber \\
    x[0] &= x_0,
\end{align}
where
\begin{equation}
    A_j = \begin{cases}A -\textnormal{diag}(\alpha_1,\ldots,\alpha_n) & \text{if} \; j=0\\ -D(\alpha,j+1) & \text{if} \; j \geq 1 \end{cases}.
\end{equation}

In the subsequent discussion, we will consider a truncation of the last $p$ temporal components of~\eqref{eq:frac_model}. Defining
\begin{equation}
\label{eq:aug_st_vec}
    \tilde{x}[k] = \begin{bmatrix} x[k] \\ x[k-1] \\ \vdots \\ x[k-p+1] \end{bmatrix}
\end{equation}
as the \emph{augmented} state vector and assuming that the system is \emph{causal}, i.e., the state and disturbances are all considered to be zero before the initial time (i.e., $x[k] = 0$ and $w[k] = 0$ for all $k < 0$), we have
\begin{align}
\label{eq:p_aug_LTI}
    \tilde{x}[k+1] &= \underbrace{\begin{bmatrix} A_0 & \ldots & A_{p-2} & A_{p-1}\\ I & \ldots & 0 & 0\\ \vdots & \ddots & \vdots & \vdots\\ 0 & \ldots & I & 0\end{bmatrix}}_{\tilde{A}}\tilde{x}[k] + \underbrace{\begin{bmatrix}I \\ 0\\ \vdots \\ 0\end{bmatrix}}_{\tilde{B}^w}w[k] \nonumber \\ &= \tilde{A}\tilde{x}[k] + \tilde{B}^w w[k],
\end{align}
for all $k \geq 0$. Note that~\eqref{eq:p_aug_LTI} is an LTI system model, which we refer to as the \emph{$p$-augmented LTI approximation} of~\eqref{eq:frac_model}.

\subsection{Two-level iterative bisection scheme to identify the parameters of a DT-FODS}

Having established the $p$-augmented LTI approximation of a DT-FODS in~\eqref{eq:p_aug_LTI}, we will now use a two-level iterative \mbox{bisection-like} approach to identify the spatial and temporal parameters of the \mbox{DT-FODS} in~\eqref{eq:frac_model}. In particular, we start by noting the fact that for the \mbox{Gr\"unwald-Letnikov} definition of the fractional derivative provided in~\eqref{eq:frac_deriv}, $\alpha_i=1$ and $\alpha_i=-1$ can be interpreted, respectively, to be the discretized version of the derivative and the integral for $1 \leq i \leq n$, as defined in the sense of ordinary calculus.

In order to proceed with a bisection-like approach to identify $\{ \alpha_i \}_{i=1}^n$ and $\tilde{A}$, we first fix the endpoints of the search space for $\alpha_i$ to be $\underline{\alpha_i} = -1$ and $\overline{\alpha_i}=1$ for $1 \leq i \leq n$. We also calculate the value of $\alpha_{c,i} = (\underline{\alpha_i}+\overline{\alpha_i})/2$. Now, given the values of $\underline{\alpha_i}, \overline{\alpha_i}$, and $\alpha_{c,i}$, we calculate using the ordinary least squares (OLS) technique described in Section~\ref{sec:OLS} below, the row vectors $\underline{\tilde{a}_i}, \overline{\tilde{a_i}}$, and $\tilde{a}_{c,i}$, respectively, that guide the evolution of the states in the $p$-augmented LTI approximation
\begin{equation}
    \tilde{x}_i[k+1] = \tilde{a}_i \tilde{x}_i[k] + \tilde{b}^w_i w_i[k],
\end{equation}
where $\tilde{a}_i = \underline{\tilde{a}_i}$ when $\alpha_i = \underline{\alpha_i}$, $\tilde{a}_i = \overline{\tilde{a}_i}$ when $\alpha_i = \overline{\alpha_i}$, and $\tilde{a}_i = \tilde{a}_{c,i}$ when $\alpha_i = \alpha_{c,i}$ with $\tilde{b}_i^w$ being obtained by extracting the $i$-th row of $\tilde{B}^w$ for $1 \leq i \leq n$.

Next, we propagate the dynamics according to the obtained values of the parameters $\tilde{a}_i$ and calculate the mean squared error (MSE) between the states obtained as a result of the estimated $\tilde{a}_i$'s and the observed states. If the MSE is lesser corresponding to the $\underline{\alpha_i}$ case, then we set $\overline{\alpha_i} = \alpha_{c,i}$. If the MSE is lesser corresponding to the $\overline{\alpha_i}$ case, then we set $\underline{\alpha_i} = \alpha_{c,i}$. This approach is repeated until $\lvert \overline{\alpha_i} - \underline{\alpha_i} \rvert$ is lesser than a certain pre-specified tolerance $\varepsilon$. Algorithm~\ref{alg:sysid} summarizes the procedure of determining the spatial and temporal components of a DT-FODS using the two-level iterative bisection-like approach we have outlined above.

Empirically, numerical and experimental evidence suggests that the computation of the temporal parameters of a DT-FODS using, e.g., a wavelet-like technique described in~\cite{flandrin1992wavelet}, does not directly depend on the number of samples or observations used for the aforementioned estimation procedure. Empirical evidence suggests that a small number of samples (usually $30$ to $100$) suffice in order to compute $\{ \alpha_i \}_{i=1}^n$. Therefore, for the estimation of the temporal components of a DT-FODS, we specify the iteration complexity of the bisection-like process and then, in Section~\ref{sec:sampcomp}, we investigate the finite-sample complexity of computing the spatial parameters using a least squares approach.

Apropos the above discussion, in the next result, we certify the iteration complexity of the bisection method to find the spatial and temporal parameters of a DT-FODS.

\begin{theorem}
The bisection-based technique detailed above to find the temporal components of a DT-FODS is minmax optimal and the number $\nu$ of iterations needed in order to achieve a certain specified tolerance $\varepsilon$ when this technique is used is bounded above by
\begin{equation}
    \nu \leq \left\lceil \log_2 \left( \frac{2}{\varepsilon} \right) \right\rceil.
\end{equation}
\end{theorem}

\begin{proof}
See~\cite{sikorski1982bisection} for a proof.
\end{proof}

\begin{algorithm}
\caption{Learning the parameters of a DT-FODS}
\begin{algorithmic}[1]
\label{alg:sysid}
\FOR{$i=1$ to $n$}
\STATE{Initialize $\underline{\alpha_i}=-1$, $\overline{\alpha_i}=1$, and tolerance $\varepsilon$.}
\STATE{Calculate $\alpha_{c,i} = (\underline{\alpha_i}+\overline{\alpha_i})/2$.}
\STATE{Given the above values of $\underline{\alpha_i}$, $\overline{\alpha_i}$, and $\alpha_{c,i}$, find, using the ordinary least squares (OLS) method, the row vectors $\underline{\tilde{a}_i}, \overline{\tilde{a_i}}$, and $\tilde{a}_{c,i}$, respectively, that guide the evolution of the states in the $p$-augmented LTI approximation $\tilde{x}_i[k+1] = \tilde{a}_i \tilde{x}_i[k] + \tilde{b}^w_i w_i[k]$.}
\STATE{Propagate the dynamics according to the obtained OLS estimates and calculate the mean squared error (MSE) between the propagated states and the observed state trajectory.}
\IF{MSE is lesser for the $\underline{\alpha_i}$ case}
\STATE{Set $\overline{\alpha_i} = \alpha_{c,i}$.}
\ELSIF{MSE is lesser for the $\overline{\alpha_i}$ case}
\STATE{Set $\underline{\alpha_i} = \alpha_{c,i}$.}
\ENDIF
\STATE{Terminate if $\lvert \overline{\alpha_i} - \underline{\alpha_i} \rvert < \varepsilon$, else return to step 3.}
\ENDFOR
\end{algorithmic}
\end{algorithm}


\subsection{The ordinary least squares (OLS) method to identify the spatial parameters of a DT-FODS}
\label{sec:OLS}

Having outlined the details of the bi-level iterative bisection-like scheme to identify the spatial and temporal components of a DT-FODS, we now delve into the problem of identifying the spatial parameters using a least squares-like approach. We start with the $p$-augmented LTI model of~\eqref{eq:p_aug_LTI}, i.e.,
\begin{equation}
    \tilde{x}[k+1] = \tilde{A}\tilde{x}[k] + \tilde{B}^w w[k].
\end{equation}
The OLS method then outputs the matrix $\underline{\tilde{A}}[K]$ as the solution of the following optimization problem
\begin{equation}
    \underline{\tilde{A}}[K] \coloneqq \argmin_{\tilde{A} \in \mathbb{R}^{d \times d}} \sum_{k=1}^K \frac{1}{2} \| \tilde{x}[k+1] - \tilde{A} \tilde{x}[k] \|_2^2,
\end{equation}
by observing the state trajectory of~\eqref{eq:p_aug_LTI}, i.e., $\{ x[0], x[1], \ldots, x[K+1] \}$, and the process noise $w[k]$ being i.i.d. zero-mean Gaussian.


\subsection{Finite-sample complexity guarantees for OLS system identification of the spatial parameters of a DT-FODS}
\label{sec:sampcomp}

Prior to characterizing the sample complexity of the OLS method for the $p$-augmented LTI approximation of the \mbox{DT-FODS}, we define a few quantities of interest.
\begin{definition}
\label{def:gramian}
The \emph{finite-time controllability Gramian} of the approximated system~\eqref{eq:p_aug_LTI}, $W_t$, is defined by
\begin{equation}
    W_t \coloneqq \sum_{j=0}^{t-1} \tilde{A}^j (\tilde{A}^j)^{\mathsf{T}}.
\end{equation}
Intuitively, the controllability Gramian gives a quantitative measure of how much the system is excited when induced by the process noise $w[k]$ acting as an input to the system.
\end{definition}
\begin{definition}
Given a symmetric matrix $A \in \mathbb{R}^{d \times d}$, we define $\lambda_{\max}(A)$ and $\lambda_{\min}(A)$ to denote, respectively, the maximum and minimum eigenvalues of the matrix $A$.
\end{definition}
\begin{definition}
For any square matrix $A \in \mathbb{R}^{d \times d}$, the \emph{spectral radius} of the matrix $A$, $\rho(A)$, is given by the largest absolute value of its eigenvalues.
\end{definition}
\begin{definition}
The \emph{operator norm} of a matrix is denoted by $\| \cdot \|_{\mathrm{op}}$.
\end{definition}



We then have our first result that characterizes the sample complexity of the above OLS method for the DT-FODS approximation.
\begin{theorem}[\cite{simchowitz2018learning}]
\label{thm:samp_comp}
Fix $\delta \in (0,1/2)$ and consider the \mbox{$p$-augmented} system in~\eqref{eq:p_aug_LTI}, where $\tilde{A} \in \mathbb{R}^{d \times d}$ is a marginally stable matrix (i.e., $\rho(\tilde{A}) \leq 1$) and $w[k] \thicksim \mathcal{N}(0,\sigma^2 I)$. Then, there exist universal constants $c,C > 0$ such that,
\begin{align}
\label{eq:bound_without_inputs}
    \mathbb{P}\Bigg[ \left\| \underline{\tilde{A}}[K] - \tilde{A} \right\|_{\mathrm{op}} &\leq  \frac{C}{\sqrt{K \lambda_{\min }\left(W_{k}\right)}} \nonumber \\ \times &\sqrt{d \log \left( \frac{d}{\delta} \right) +\log \operatorname{det}\left(W_{K} W_{k}^{-1}\right)} \Bigg] \nonumber \\ & \geq 1 - \delta,
\end{align}
for any $k$ such that
\begin{equation}
    \frac{K}{k} \geq c\left(d \log \left( \frac{d}{\delta} \right) + \log \operatorname{det}\left(W_{K} W_{k}^{-1}\right)\right)
\end{equation}
holds.
\end{theorem}

\begin{proof}
The proof of the theorem is a consequence of Theorem 2.4 in~\cite{simchowitz2018learning}. In general, if we consider the time series $(X[t],Y[t])_{t \geq 1}$, where $Y[t] = A_\star X[t] + \eta [t]$, with $Y[t], \eta[t] \in \mathbb{R}^n, X[t] \in \mathbb{R}^d$, and $A_\star \in \mathbb{R}^{n \times d}$, the OLS estimate is obtained by solving the optimization problem
\begin{equation}
\widehat{A}(T) \coloneqq \argmin _{A \in \mathbb{R}^{n \times d}} \sum_{t=1}^{T} \frac{1}{2}\left\|Y[t] - A X[t]\right\|_{2}^{2},
\end{equation}
where $T$ is the observation horizon. We consider \mbox{$\mathcal{F}_t \coloneqq \sigma (\eta[0], \eta[1], \ldots, \eta[t], X[1], \ldots, X[t])$} to be a filtration generated by the states and the noise, and that $\eta[t] \vert \mathcal{F}_{t-1}$ to be zero-mean and $\sigma^2$-sub-Gaussian (or, in other words, sub-Gaussian with a variance proxy $\sigma^2$), i.e., \mbox{$\mathbb{E} [ \exp (\lambda \eta[t] \vert \mathcal{F}_t) ] \leq \exp( \lambda^2 \sigma^2 / 2)$} for all $\lambda \in \mathbb{R}$.

If we fix $\epsilon, \delta \in (0,1)$, $T \in \mathbb{N}$, and there exist $W_{\mathrm{sb}}$ and $\overline{W}$ such that $0 \prec W_{\mathrm{sb}} \preceq \overline{W}$ (where the $W$'s are the controllability Gramians defined as in Definition~\ref{def:gramian}), then for a random sequence $(X[t],Y[t])_{t \geq 1} \in ( \mathbb{R}^d \times \mathbb{R}^n )^T$ that satisfies 
\begin{enumerate}
    \item[(a)] $Y[t] = A_\star X[t] + \eta [t]$ with $\eta[t] \vert \mathcal{F}_t$ being zero-mean and $\sigma^2$-sub-Gaussian,
    \item[(b)] $\{ X[1],\ldots,X[T] \}$ satisfies the $(k,W_{\mathrm{sb}},p)$-martingale small-ball condition (Definition 2.1 in~\cite{simchowitz2018learning}), and,
    \item[(c)] $\mathbb{P} \left[ \sum_{t=1}^T X[t] X[t]^{\mathsf{T}} \npreceq T \overline{W} \right] \leq \delta$,
\end{enumerate}
we have that if
\begin{equation}
T \geq \frac{10 k}{p^{2}}\left(\log \left(\frac{1}{\delta}\right)+2 d \log \left(\frac{10}{p} \right)+\log \operatorname{det}\left(\overline{W} W_{\mathrm{sb}}^{-1}\right)\right),
\end{equation}
then
\begin{dmath}
\mathbb{P} \left[ \left\| \widehat{A}(T) - A_\star \right\|_{\mathrm{op}} \\ \leq \frac{90\sigma}{p} \sqrt{\frac{n + d \log  \frac{10}{p} + \log \det \overline{W} W_{\mathrm{sb}}^{-1} +\log  \frac{1}{\delta}  }{T \lambda_{\min} (W_{\mathrm{sb}})}} \right ] \geq 1 - 3\delta.
\end{dmath}
Further, for the specific case considered in this paper, we know that the process noise satisfies the $\sigma^2$-sub-Gaussian condition. The intuition behind this follows easily from the fact that if a random variable $X$ has the distribution $\mathcal{N}(0,\sigma^2)$, then $\mathbb{E}(\exp (\lambda X)) = \exp(\lambda^2 \sigma^2/2)$ for all $\lambda \in \mathbb{R}$, and thus $X$ is $\sigma^2$-sub-Gaussian, with a similar argument holding for Gaussian random vectors as well. We further note that the observed length of our state trajectory is $K$. Hence, we can write
\begin{align}
    \mathbb{P}\Bigg[&\mathbf{X}^{\mathsf{T}} \mathbf{X} \npreceq \frac{d \sigma^{2}}{\delta} K W_K\Bigg] \nonumber \\ =&\mathbb{P}\left[\lambda_{\max }\left(\left(K W_K\right)^{-1 / 2} \mathbf{X}^{\mathsf{T}} \mathbf{X}\left(K W_K\right)^{-1 / 2}\right) \geq \frac{d \sigma^{2}}{\delta}\right] \nonumber \\
    & \stackrel{(\bigtriangleup)}{\leq} \frac{\delta}{d \sigma^{2}} \cdot \mathbb{E}\left[\lambda_{\max }\left(\left(K W_K\right)^{-1 / 2} \mathbf{X}^{\mathsf{T}} \mathbf{X}\left(K W_K\right)^{-1 / 2}\right)\right] \nonumber \\
    & \leq \frac{\delta}{d \sigma^{2}} \cdot \mathbb{E}\left[\operatorname{tr}\left(\left(K W_K \right)^{-1 / 2} \mathbf{X}^{\mathsf{T}} \mathbf{X}\left(K W_K \right)^{-1 / 2}\right)\right] \nonumber \\
    & \stackrel{(\bigtriangledown)}{\leq} \delta,
\end{align}
where the rows of $\mathbf{X} \in \mathbb{R}^{K \times d}$ are formed using $\tilde{x}[k]$ for each $k$. The inequality $(\bigtriangleup)$ is a consequence of Markov's inequality and the inequality $(\bigtriangledown)$ is a consequence of $\mathbb{E} \left[ \mathbf{X}^{\mathsf{T}} \mathbf{X} \right] = \sigma^2 \sum_{j=1}^K W_j \preceq \sigma^2 K W_K$ and the linearity of the trace operator. The proof of the theorem follows from setting $\overline{W} = \frac{d \sigma^2}{\delta} W_K$, the equality $\log \operatorname{det}\left(\left(\frac{d}{\delta} \sigma^{2} W_K \right) \left(\sigma^{2} W_{\lfloor k / 2\rfloor}\right)^{-1} \right) = d \log ( d / \delta) +\log \operatorname{det}\left( W_K W_{\lfloor k / 2 \rfloor}^{-1}\right)$, and the fact that $\tilde{x}[k]$ satisfies the $(k,\sigma^2 W_{\lfloor k / 2 \rfloor},\frac{3}{20})$-block martingale small-ball condition (Proposition 3.1 of~\cite{simchowitz2018learning}).
\end{proof}

\begin{remark}
We note here that although the operator norm parameter estimation error in~\eqref{eq:bound_without_inputs} is stated in terms of $\tilde{A}$, the operator norm errors associated with the matrices $A_0, A_1, \ldots, A_{p-1}$ are strictly lesser compared to $\left\| \underline{\tilde{A}}[K] - \tilde{A} \right\|_{\mathrm{op}}$, since $A_0, A_1, \ldots, A_{p-1}$ are submatrices of $\tilde{A}$, and for any operator norm, the operator norm of a submatrix is upper bounded by one of the whole matrix (see Lemma A.9 of~\cite{foucart2013} for a proof).
\end{remark}

\begin{remark}
A finite-sample complexity bound similar to the one presented in Theorem~\ref{thm:samp_comp} can also be derived when we consider the ordinary least squares identification of the spatial parameters of a DT-FODS with inputs. In the context of our work, an external input may correspond to an exogenous electrical stimulus or an actuation or control signal. Therefore, within the purview of epileptic seizure mitigation using intracranial EEG data, the objective of the former is to suppress the overall length or duration of an epileptic seizure, thus steering the state of the neurophysiological system in consideration away from seizure-like activity using a control strategy like model predictive control~\cite{chatterjee2020fractional}.
\end{remark}

\section{Simulation Results on Real-Life Intracranial EEG Data}
\begin{figure}
    \centering
    \includegraphics[width=0.5\textwidth]{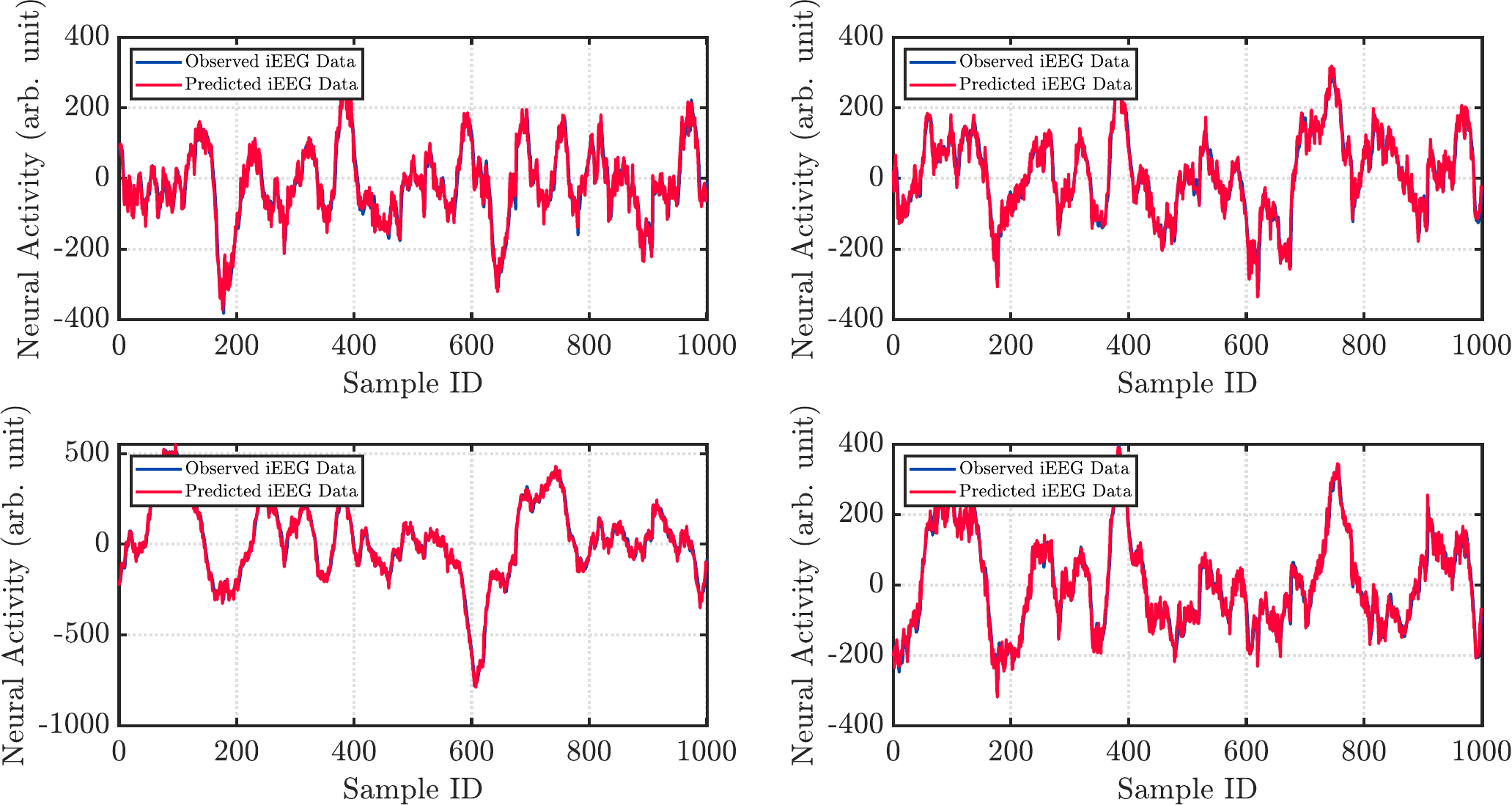}
    \caption{Performance of our approach on real-life intracranial EEG data.}
    \label{fig:eeg_sysid}
\end{figure}

\begin{figure}
    \centering
    \includegraphics[width=0.5\textwidth]{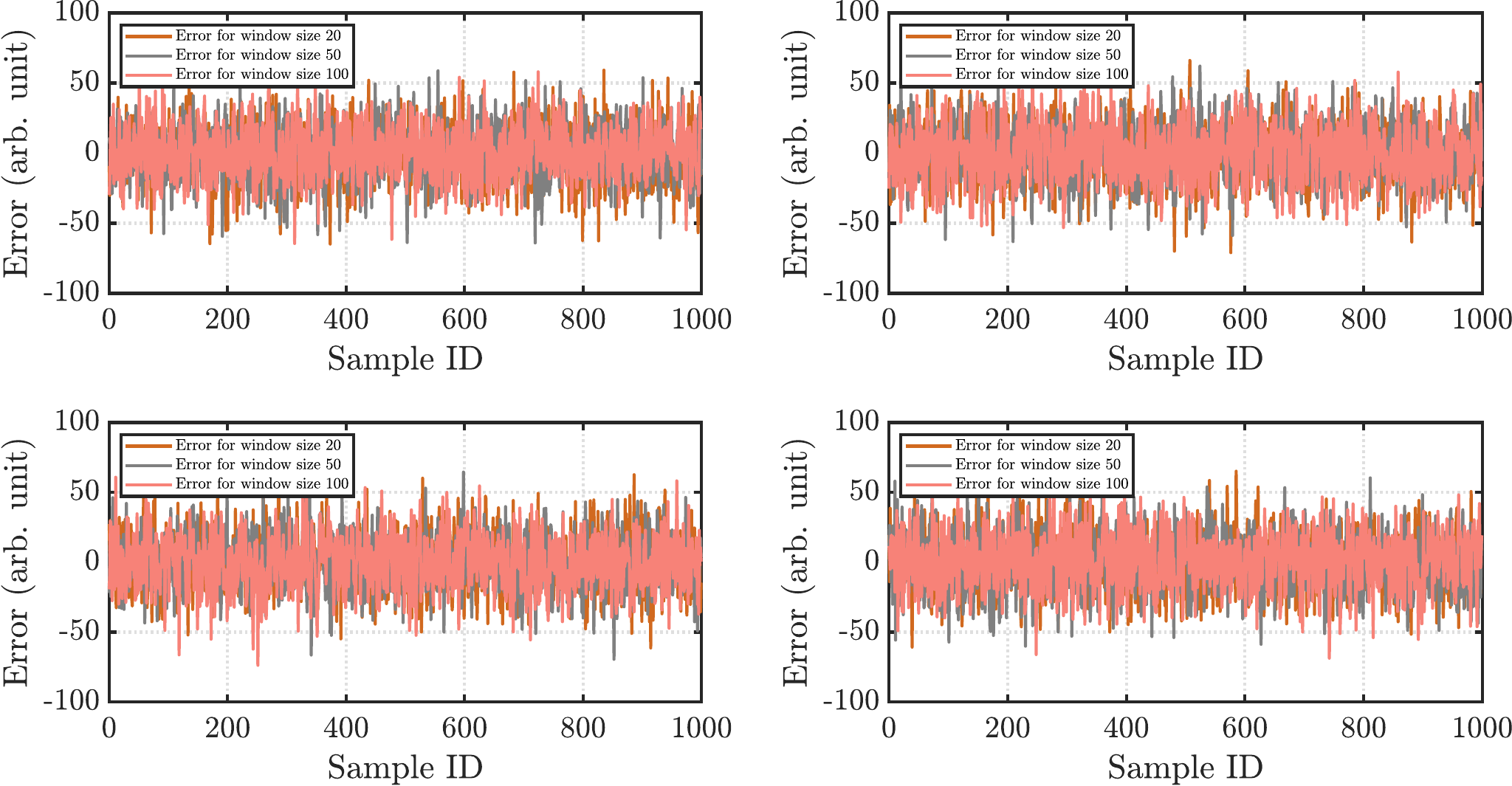}
    \caption{Variation of the error of the least squares prediction with respect to the observed data, with varying window sizes in the least squares optimization problems.}
    \label{fig:eeg_sysid_many}
\end{figure}

We present some preliminary results regarding the performance of the above approach on \mbox{DT-FODS}. Specifically, we use $1000$ noisy measurements taken from $4$ channels of an intracranial electroencephalographic (iEEG) signal which records the brain activity of subjects undergoing epileptic seizures. The signals were recorded and digitized at a sampling rate of $512$ Hz at the Hospital of the University of Pennsylvania, Philadelphia, PA. Subdural grid and strip electrodes were placed at specialized locations (dictated by a multidisciplinary team of neurologists, neurosurgeons, and a radiologist), with the electrodes themselves consisting of linear and two-dimensional arrays spanning $2.3$ mm in diameter and having a inter-contact spacing of $10$ mm~\cite{khambhati2015dynamic,ashourvan2020model}. All of the least squares optimization problems are solved using \texttt{CVX}~\cite{cvx},~\cite{grant08} with the aid of a \mbox{window-based} approach using a finite subset of the entire range of measurements. This is done because the time series under consideration is nonlinear, and it is not possible to characterize the entire gamut of measurements using very few parameters. Figure~\ref{fig:eeg_sysid} shows the performance of our method on the above data. Additionally, we also show in Figure~\ref{fig:eeg_sysid_many} the variation of the error of the least squares predictions with respect to the observed data, with varying window sizes in the least squares optimization problems. We see that the identified system parameters are able to predict the system states fairly closely, thus demonstrating that our approach can be used to learn the system parameters of a DT-FODS.

\section{Conclusions and Future Work}

In this paper, we present a framework that enables us to learn the spatial and temporal parameters of a \mbox{DT-FODS}. We provide non-asymptotic finite-sample complexity guarantees for the identification of the spatial components of autonomous and actuated \mbox{DT-FODS} using least squares as well as iteration complexity guarantees for the identification of the temporal components of a DT-FODS. We also show the efficacy of our proposed approach in identifying the system parameters of a DT-FODS that models iEEG signals characterizing the behavior of real-life subjects undergoing epileptic seizures.

There are a number of interesting directions which can be taken from here. A particular case that piques our interest is system identification using the Expectation-Maximization (EM) algorithm. Although approaches using the EM algorithm for linear~\cite{gibson2005robust} as well as nonlinear system identification~\cite{schon2011system} have existed in the literature for a while now, one immediately notices that there is a \mbox{long-standing} problem in characterizing theoretical robustness guarantees for the same. We are inspired by the preliminary analyses of \mbox{finite-sample} robustness guarantees for EM in~\cite{wu2016convergence,balakrishnan2017statistical,yan2017convergence}, and, therefore, we wish to characterize the sample complexity in identifying LTI as well as \mbox{fractional-order} systems using the EM algorithm. Furthermore, the system identification of \mbox{fractional-order} systems is an extremely \mbox{under-explored} field in general, with a lack of a systematic and unified theory, with some preliminary approaches utilizing wavelets~\cite{flandrin1992wavelet}, \mbox{frequency-domain} techniques~\cite{adams2006fractional,dzielinski2011identification}, or a sequential combination of wavelets and EM~\cite{gupta2018dealing}. Future work will focus on chalking out a general theory of identifying certain classes of \mbox{fractional-order} systems.

\balance
\bibliographystyle{IEEEtran}
\bibliography{IEEEabrv,mybibfile}

\end{document}